\documentclass[11pt]{amsart}
\usepackage{amsmath,amssymb,amsthm}
\usepackage[margin=1.1in]{geometry}
\usepackage{enumitem}

\newtheorem{theorem}{Theorem}[section]
\newtheorem{proposition}[theorem]{Proposition}
\newtheorem{lemma}[theorem]{Lemma}
\newtheorem{corollary}[theorem]{Corollary}
\theoremstyle{remark}
\newtheorem{remark}[theorem]{Remark}

\newcommand{\T}{\mathbb{T}}
\newcommand{\R}{\mathbb{R}}
\newcommand{\Z}{\mathbb{Z}}
\newcommand{\Q}{\mathbb{Q}}
\newcommand{\C}{\mathbb{C}}
\newcommand{\Homeo}{\operatorname{Homeo}}

\newcommand{\Fix}{\operatorname{Fix}}
\newcommand{\Aff}{\operatorname{Aff}}
\newcommand{\GL}{\operatorname{GL}}
\newcommand{\SL}{\operatorname{SL}}

\title[Topological centralizers of perturbations of ergodic toral automorphisms]{Topological centralizers of perturbations of ergodic toral automorphisms with two dimensional center}

\author{Boris Petkovi\'c}
\address{University of Banja Luka Faculty of Natural Sciences and Mathematics}
\email{boris.petkovic@pmf.unibl.org}
\date{August 2026}

\begin{document}

\begin{abstract}
Let $A$ be an ergodic linear automorphism of the torus $\T^N$ whose center space has dimension two. We prove a dichotomy result that for every $f \in \mathcal{U}$, where $\mathcal{U}$ is the $C^{22}$ neighborhood of $A$ inside volume preserving diffeomorphisms on $\T^N$, exactly one of two alternatives holds. Either $f$ has the accessibility property, or $f$ is topologically conjugate to $A$ by a homeomorphism homotopic to the identity which simultaneously conjugates every homeomorphism commuting with $f$ to an affine automorphism of the torus. No irreducibility of the characteristic polynomial of $A$ is assumed. The alternatives exclude each other
because a partially hyperbolic diffeomorphism topologically conjugate to $A$ is never accessible,
so within $\mathcal{U}$ the failure of accessibility is equivalent to topological conjugacy to $A$. In the second case the centralizer of $f$ in the group of homeomorphisms of the torus is isomorphic to the group of affine transformations commuting with $A$. When the characteristic polynomial of $A$ is irreducible we show that this group is a finite extension of the unit group of an order in the number field generated by an eigenvalue of $A$, hence virtually free abelian of rank $r_1 + r_2 - 1$. The characteristic polynomial is allowed to be reducible, in which case the same description holds with the linear parts ranging over the centralizer of $A$ in $\GL(N,\Z)$. Our method uses an elementary density property of the projection of the lattice to the center space, which replaces irreducibility and which we establish for every ergodic automorphism with no restriction on the dimension of the center space.
\end{abstract}

\maketitle

\section{Introduction}

Let $\T^N = \R^N / \Z^N$ and let a matrix $A \in \SL(N,\Z)$ act on $\T^N$ in the canonical way. By a classical result of Halmos \cite{Ha} the induced automorphism is ergodic with respect to Lebesgue measure if and only if no eigenvalue of $A$ is a root of unity. A volume preserving diffeomorphism is called stably ergodic if every volume preserving diffeomorphism sufficiently close to it in an appropriate topology is ergodic. Hyperbolic autmorphisms are stably ergodic by the theory of Anosov \cite{An}, and the question whether every ergodic linear automorphism of the torus is stably ergodic, raised by Hirsch, Pugh and Shub, fits into the program of Pugh and Shub \cite{PS} on stable ergodicity of partially hyperbolic systems. A cornerstone of that program is the theorem of Burns and Wilkinson \cite{BW}, by which accessibility together with center bunching implies ergodicity for volume preserving partially hyperbolic diffeomorphisms. Rodriguez Hertz \cite{RH} answered the question affirmatively for pseudo-Anosov automorphisms, namely those ergodic $A$ whose characteristic polynomial $p_A$ is irreducible over the integers and is not a polynomial in $t^n$ for any $n \geq 2$, whose center space $E^c$, the sum of the generalized eigenspaces associated to the eigenvalues of modulus one, has dimension two. His theorem gives stable ergodicity in the $C^5$ topology when $N \geq 6$ and in the $C^{22}$ topology when $N = 4$. \\

Very recently Argentieri and Ulliana \cite{AU} removed the algebraic restrictions entirely from \cite{RH} for two dimensional center. They proved that every ergodic linear automorphism of $\T^N$ with $\dim E^c = 2$ is stably ergodic in the $C^{22}$ topology, with no irreducibility assumption on $p_A$. The key new ingredient is a minimality criterion for saturated invariant sets which replaces the only step of \cite{RH} where the pseudo-Anosov condition is used. Their theorem covers in particular all ergodic automorphisms in dimension $N \leq 5$ and $N = 7$, and all ergodic automorphisms in dimensions $N=6$ and $N=9$ none of whose eigenvalues is a Salem number, since in these cases a non Anosov ergodic automorphism automatically has two dimensional center. \\

A remarkable feature of the proofs in \cite{RH} and \cite{AU} is a rigidity alternative. For $f$ in the relevant neighborhood of $A$ either $f$ has the accessibility property, which means that any two points of the torus can be joined by a path consisting of finitely many arcs each of which is contained in a single strong stable or strong unstable leaf of $f$, or the strong stable and strong unstable bundles of $f$ are jointly integrable. In the jointly integrable case Rodriguez Hertz proved, for pseudo-Anosov $A$, that $f$ is topologically conjugate to $A$. In \cite{AU} the jointly integrable case is treated by linearizing the associated holonomy action through the theorems of Herman and Moser and deriving a $C^1$ conjugacy between the strong foliations of $f$ and the linear foliations of $A$, which transfers essential accessibility and hence ergodicity, and the full dynamical conjugacy is not stated there. Our first contribution, Proposition \ref{prop:conjugacy} below, is to carry out the construction of the dynamical conjugacy for the whole class of \cite{AU}. The passage from the pseudo-Anosov case to the general two dimensional center case requires replacing the use of irreducibility at one point of the argument of \cite{RH}, and we do so by an elementary density lemma for the projection of the lattice $\Z^N$ to the center space, Lemma \ref{lem:density} below. We establish that lemma for every ergodic automorphism of the torus with nontrivial center space, with no restriction on the dimension of the center. \\

The purpose of this replacement is a statement about commuting maps. For a homeomorphism $f$ of a manifold $M$ let $Z(f)$ denote the centralizer of $f$ in $\Homeo(M)$, that is the group of all homeomorphisms $g$ of $M$ with $f \circ g = g \circ f$. Smale \cite{Sm} asked whether a generic diffeomorphism has trivial centralizer, and Bonatti, Crovisier and Wilkinson \cite{BCW} proved that this is the case in the $C^1$ topology. For partially hyperbolic systems of algebraic origin a program of centralizer rigidity has emerged in the work of Damjanovi\'c, Wilkinson and Xu \cite{DWX}, continued in \cite{DWWX}, in which the centralizer of a perturbation is either small or the perturbation is itself of algebraic nature. For hyperbolic $A$ the centralizer of every $C^1$ small perturbation is understood classically, since structural stability provides a conjugacy to $A$ for every such perturbation and the theorem of Adler and Palais \cite{AP} then identifies the centralizer with the group of affine transformations commuting with $A$. The theorem below is the exact analogue for two dimensional center, with the accessibility alternative playing the role that structural stability plays in the hyperbolic case. \\

We fix the following notation. Given $A \in \GL(N,\Z)$ with no eigenvalue equal to one, let
\[
\Fix(A) = \{ w \in \T^N \, \colon \, A w = w \},
\]
a finite subgroup of $\T^N$ of order $|p_A(1)|$, and let
\[
\Aff_A = \{ x \mapsto B x + w \, \colon \, B \in \GL(N,\Z), \, A B = B A, \, w \in \Fix(A)\}
\]

denote a subgroup of the group of affine transformations of $\T^N$. Throughout the paper $A \in \SL(N,\Z)$ denotes an ergodic linear automorphism of $\T^N$ with $\dim E^c = 2$, and $\mathcal{U}$ denotes the neighborhood of $A$ in the space of volume preserving $C^{22}$ diffeomorphisms of $\T^N$ provided by the theorem of Argentieri and Ulliana, recalled as Theorem \ref{thm:au} below.

\begin{theorem}\label{thm:main}
Let $f \in \mathcal{U}$. Then exactly one of the following two assertions holds.
\begin{enumerate}
\item The diffeomorphism $f$ has the accessibility property.
\item There is a homeomorphism $H$ of $\T^N$ homotopic to the identity such that $H \circ f \circ H^{-1} = A$ and such that for every homeomorphism $g$ of $\T^N$ commuting with $f$ the map $H \circ g \circ H^{-1}$ is the affine transformation $x \mapsto B x + w$, where $B \in \GL(N,\Z)$ is the automorphism induced by $g$ on the first homology group of the torus, which necesseraly satisfies $AB=BA$, and $w \in \Fix(A)$. In particular conjugation by $H$ is a group isomorphism between $Z(f)$ and $\Aff_A$.
\end{enumerate}
\end{theorem}

We emphasize that in the second alternative the commuting map $g$ is an arbitrary homeomorphism. No smoothness, no volume preservation and no proximity of $g$ to a linear model are required, and the characteristic polynomial of $A$ may be reducible. The exclusivity of the two alternatives rests on Proposition \ref{prop:exclusive} below, which shows that a partially hyperbolic diffeomorphism topologically conjugate to $A$ is never accessible. For volume preserving smooth perturbations sufficiently close to the automorphism this equivalence between the failure of accessibility and topological conjugacy was observed by Sandfeldt \cite{Sa} as a consequence of \cite{RH}, and Proposition \ref{prop:exclusive} provides a general form of the statement, with an elementary self contained proof, valid for an arbitrary partially hyperbolic diffeomorphism conjugate to $A$ by an arbitrary homeomorphism, without closeness to $A$, volume preservation, or hypotheses on the characteristic polynomial. Together with Proposition \ref{prop:conjugacy} this gives the following characterization.

\begin{corollary}\label{cor:equivalence}
Let $f \in \mathcal{U}$. The following three conditions are equivalent.
\begin{enumerate}
\item The diffeomorphism $f$ does not have the accessibility property.
\item The diffeomorphism $f$ is topologically conjugate to $A$.
\item The strong stable and strong unstable bundles of $f$ are jointly integrable.
\end{enumerate}
\end{corollary}

\begin{remark}\label{rem:regularity-roles}
The regularity in Theorem \ref{thm:main} plays an asymmetric role. The $C^{22}$ hypothesis, or $C^5$ in the pseudo-Anosov case with $N \geq 6$, is what establishes the accessibility alternative through \cite{AU} and \cite{RH}, and it therefore determines which case of the dichotomy $f$ lies in. The conclusions themselves are purely topological. The group $Z(f)$ consists of arbitrary homeomorphisms, the conjugacy $H$ is only continous, and the affine rigidity behind the second case, the theorem of Adler and Palais, is a topological statement, as is Proposition \ref{prop:exclusive}. Once the case of the dichotomy is known, the topological picture is complete and requires no further smoothness. This contrasts with the smooth centralizer classification of \cite{Sa}, which asks which $C^\infty$ diffeomorphisms commute with $f$ and accordingly works in the smooth category throughout.
\end{remark}

The group $\Aff_A$ has a transparent algebraic structure, described in Section 3. Combining it with Theorem \ref{thm:main} yields the following consequences. Here $\lambda$ denotes an eigenvalue of $A$ generating the field $K = \Q(\lambda)$ with $r_1$ real embeddings and $r_2$ pairs of complex conjugate embeddings.

\begin{corollary}\label{cor:structure}
Let $f \in \mathcal{U}$ and suppose that $f$ does not have the accessibility property. Then $Z(f)$ contains a finite normal subgroup isomorphic to $\Fix(A)$ with quotient isomorphic to the centralizer of $A$ in $\GL(N,\Z)$, and this extension splits. If $p_A$ is irreducible the quotient is isomorphic to the unit group of an order in $K$, so $Z(f)$ is virtually free abelian of rank $r_1 + r_2 - 1$. If $N = 4$ then $p_A$ is automatically irreducible with $r_1 = 2$ and $r_2 = 1$, so $Z(f)$ is virtually isomorphic to $\Z^2$ and $f$ embeds in a rank two abelian group of commuting homeomorphisms of $\T^4$.
\end{corollary}

\begin{corollary}\label{cor:criterion}
Let $f \in \mathcal{U}$. If $f$ commutes with a homeomorphism of infinite order that acts trivially on the first homology group of $\T^N$, or if $Z(f)$ is uncountable, then $f$ has the accessibility property. If moreover $p_A$ is irreducible and $Z(f)$ is not virtually abelian, then $f$ has the accessibility property.
\end{corollary}

\begin{corollary}\label{cor:dynamics}
Let $f \in \mathcal{U}$ be without the accessibility property and let $g$ be a homeomorphism of $\T^N$ commuting with $f$, with induced matrix $B$ on the first homology group. If no eigenvalue of $B$ is a root of unity then $g$ is topologically conjugate to the automorphism induced by $B$. Consequently $g$ is topologically mixing, the periodic points of $g$ are dense, and the topological entropy of $g$ equals the sum of $\log |\beta|$ over the eigenvalues $\beta$ of $B$ with $|\beta| > 1$, counted with multiplicity.
\end{corollary}

Two comments on the relation with previous work are in order. First, the affine rigidity used here is classical. The theorem of Adler and Palais \cite{AP} applies to every ergodic automorphism of the torus, hyperbolic or not, and Walters \cite{W1, W2} proved that ergodicity is precisely the necessary and sufficient condition and extended the statement to compact abelian groups. We include a complete proof in Section 3, both to keep the paper self contained and because the transfer lemma underlying it is reused in the uniqueness arguments of Section 5. Second, the closest result in the literature is due to Sandfeldt \cite{Sa}, who classified the possible smooth centralizers of volume preserving diffeomorphisms $C^1$ close to an ergodic irreducible toral automorphism with two dimensional isometric center, proving that the image of the smooth centralizer in the automorphism group of the first homology has rank at most $r_1 + r_2 - 1$ and that $f$ is $C^\infty$ conjugate to the automorphism whenever this rank is attained. The results of the present paper differ in several respects. We work with the full topological centralizer, allowing arbitrary commuting homeomorphisms with no smoothness or volume preservation assumptions. We compute the centralizer exactly rather than bounding its rank. We allow reducible characteristic polynomials, and the mechanism of proof is different, resting on the accessibility alternative and the theorem of Adler and Palais rather than on rigidity of higher rank abelian actions in the sense of \cite{DK}. On the other hand our conclusions require $C^{22}$ closeness and are confined to the case of the alternative where accessibility fails. However, we combine our results with those of \cite{Sa}. Corollary \ref{cor:smooth} below chains Proposition \ref{prop:exclusive} with the dichotomy of \cite{Sa} and shows that in dimension four every smooth accessible perturbation has virtually trivial smooth centralizer, while every non accessible one has topological centralizer of rank two, so the two sides of the alternative are separated by centralizers in different categories.

The paper is organized as follows. Section 2 collects the algebraic preliminaries, among them the density lemma, and the statement from \cite{AU} that we use. Section 3 establishes the affine rigidity of centralizers of ergodic automorphisms with semisimple center. Section 4 contains the construction of the topological conjugacy in the jointly integrable case. Section 5 establishes the exclusivity of the alternative and contains the proofs of Theorem \ref{thm:main} and its corollaries together with statements about the induced group actions and the behavior of accessibility classes. Section 6 discusses the limits of the method.

\section{Preliminaries}

We write $p \colon \R^N \to \T^N$ for the canonical projection and denote Lebesgue measure on the torus by $m$. For $A \in \GL(N,\Z)$ let
\[
\R^N = E^s \oplus E^c \oplus E^u
\]
be the splitting induced by $A$ into the sums of the stable, center and unstable space, respectively. We write $A^s$, $A^c$ and $A^u$ for the corresponding restrictions of $A$, and $v = v^s + v^c + v^u$ for the components of a vector $v \in \R^N$. The transpose of $A$ is denoted $A^T$.

\begin{lemma}\label{lem:linearalgebra}
Let $A \in \GL(N,\Z)$ induce an ergodic automorphism of $\T^N$ with $\dim E^c = 2$. Then the following assertions hold.
\begin{enumerate}
\item The eigenvalues of $A$ of modulus one form a single pair $\lambda$, $\bar\lambda$ with $\lambda$ nonreal and not a root of unity, each of algebraic multiplicity one.
\item The restriction $A^c$ is semisimple, it preserves a Euclidean inner product on $E^c$, and with respect to that inner product it is a rotation by an angle which is an irrational multiple of $2\pi$.
\item Every real matrix commuting with $A$ preserves $E^s$, $E^c$ and $E^u$.
\item 1 is not an eigenvalue of $A$ and $\det (A - I) = \pm p_A(1) \neq 0$.
\item For every $m \in \Z^N$ with $m \neq 0$ the orbit $\{ (A^T)^k m \, \colon \, k \in \Z \}$ is infinite, hence unbounded.
\end{enumerate}
\end{lemma}

\begin{proof}
A real number of modulus one equals $1$ or $-1$, and both of these are roots of unity. Therefore ergodicity forces every eigenvalue of modulus one to be nonreal. Nonreal eigenvalues of a real matrix occur in conjugate pairs with equal algebraic multiplicities, and $\dim E^c$ equals the sum of the algebraic multiplicities of the eigenvalues of modulus one. Since $\dim E^c = 2$ there is exactly one such pair. Each member of that pair has multiplicity one, which proves the first assertion. Multiplicity one implies that the generalized eigenspaces of $\lambda$ and $\bar\lambda$ are eigenlines, so the complexifcation of $A^c$ is diagonalizaable with eigenvalues $\lambda$ and $\bar\lambda$ of modulus one. Hence $A^c$ is conjugate over the reals to the rotation by the angle $\arg \lambda$. This angle is an irrational multiple of $2\pi$ because $\lambda$ is not a root of unity, and it preserves the Euclidean inner product obtained from the standard one by the conjugation. This proves the second assertion. A matrix commuting with $A$ commutes with every polynomial in $A$ and therefore preserves the generalized eigenspace $\ker (A - \mu I)^N$ of the complexification for every eigenvalue $\mu$. Hence it preserves the sums of these spaces over the eigenvalues of a fixed range of moduli, hence the real subspaces $E^s$, $E^c$ and $E^u$. The fourth assertion is clear since one is a root of unity, and $\det(A - I)$ equals $\pm p_A(1)$ by the definition of the characteristic polynomial. Finally, if the set $\{ (A^T)^k m \, \colon \, k \in \Z \}$ is finite,  $(A^T)^k m = m$ for some $k \geq 1$ and $m \neq 0$. Then one is an eigenvalue of $(A^T)^k$, so some eigenvalue of $A$ is a root of unity, which is a contradiction. An infinite subset of $\Z^N$ is unbounded.
\end{proof}

The next lemma is the substitute for irreducibility of $p_A$ announced in the introduction. We state and prove it for an arbitrary ergodic automorphism, with no restriction on the dimension of the center space, since the general statement requires no additional effort and could be relevant for higher dimensional center space. Equivalently, the lemma asserts that the center space of the transpose matrix contains no nonzero integer vector.

\begin{lemma}\label{lem:density}
Let $A \in \GL(N,\Z)$ induce an ergodic automorphism of $\T^N$ and suppose $E^c \neq \{0\}$. Then the set $\{ n^c \, \colon \, n \in \Z^N \}$ is dense in $E^c$.
\end{lemma}

\begin{proof}
Let $d = \dim E^c$ and let $\Gamma$ be the closure of $\{ n^c \, \colon \, n \in \Z^N \}$ in $E^c$. Since the projection $x \mapsto x^c$ is linear, $\Gamma$ is a closed subgroup of $E^c$. Suppose, for a contradiction, that $\Gamma \neq E^c$. By the structure theorem for closed subgroups of a finite dimensional real vector space there are a linear subspace $V \subseteq E^c$, equal to the connected component of the identity of $\Gamma$, and vectors $w_1, \dots, w_b \in E^c$, linearly independent modulo $V$, such that $\Gamma = V + \Z w_1 + \dots + \Z w_b$. If $\dim V + b < d$ we choose a nonzero linear functional $\xi$ on $E^c$ vanishing on $V$ and on every $w_j$, so that $\xi(\Gamma) = \{0\}$. If $\dim V + b = d$ then $b \geq 1$, since $\Gamma \neq E^c$, and we choose $\xi$ vanishing on $V$ and on $w_1, \dots, w_{b-1}$ with $\xi(w_b) = 1$, so that $\xi(\Gamma) = \Z$. In both cases $\xi \neq 0$ and $\xi(\Gamma) \subseteq c \Z$ for some $c \geq 0$.

Define the linear functional $\eta$ on $\R^N$ by $\eta(x) = \xi(x^c)$. Then $\eta$ vanishes on $E^s \oplus E^u$, the restriction of $\eta$ to $E^c$ equals $\xi$, so $\eta \neq 0$, and $\eta(\Z^N) \subseteq a \Z$. If $a = 0$ then $\eta$ vanishes on $\Z^N$, hence on all of $\R^N$, which $\Z^N$ spans, a contradiction. Therefore $a > 0$, and after replacing $\xi$ by $\xi / a$ we may assume $\eta(\Z^N) \subseteq \Z$. Setting $m = (\eta(e_1), \dots, \eta(e_N)) \in \Z^N$, where $e_1, \dots, e_N$ denotes the standard basis, we obtain $\eta(x) = \langle m, x \rangle$ for every $x \in \R^N$, with $m \neq 0$ and
\[
\langle m, v \rangle = 0 \qquad \text{for every } v \in E^s \oplus E^u .
\]

We claim that the annihilator of $E^s \oplus E^u$ with respect to the standard pairing equals the center space $E^c(A^T)$ of the transpose matrix. Let $\mu$ be an eigenvalue of $A$ and let $\nu$ be an eigenvalue of $A^T$ with $\nu \neq \mu$, and let $x$ and $y$ be vectors of the corresponding generalized eigenspaces of the complexifications, say $(A - \mu)^p x = 0$ and $(A^T - \nu)^q y = 0$. The operator $A - \nu$ restricted to $\ker (A - \mu)^p$ is invertible, because its unique eigenvalue there is $\mu - \nu \neq 0$, so $x = (A - \nu)^q z$ for some $z \in \ker(A - \mu)^p$, and
\[
\langle y, x \rangle = \langle y, (A - \nu)^q z \rangle = \langle (A^T - \nu)^q y, z \rangle = 0 .
\]
Since $A$ and $A^T$ have the same characteristic polynomial, this computation show that $E^c(A^T)$ pairs to zero with every generalized eigenspace of $A$ associated to an eigenvalue of modulus different from one, so $E^c(A^T)$ is contained in the annihilator of $E^s \oplus E^u$. Both spaces have dimension $d$, which proves the claim. Hence $m$ is a nonzero integer vector lying in $E^c(A^T)$.

Let $W$ be the linear span of the orbit $\{ (A^T)^j m \, \colon \, j \geq 0 \}$. Then $A^T W \subseteq W$, and equality holds because $A^T$ is invertible and $W$ is finite dimensional. The space $W$ is spanned by integer vectors, so $\Lambda = W \cap \Z^N$ is a lattice of full rank in $W$, and $A^T$ maps $\Lambda$ onto itself because $A^T$ and its inverse preserve $\Z^N$ and $W$. In a basis of $\Lambda$ the restriction of $A^T$ to $W$ is therefore given by a matrix $B \in \GL(d', \Z)$ with $d' = \dim W$. On the other hand $m \in E^c(A^T)$ and $E^c(A^T)$ is invariant under $A^T$, so $W \subseteq E^c(A^T)$. Every eigenvalue of $B$ is an eigenvalue of $A^T$ admitting an eigenvector in $E^c(A^T) \otimes_\R \C$, and since the generalized eigenspaces associated to distinct eigenvalues are independent, such an eigenvalue has modulus one. The characteristic polynomial of $B$ is monic with integer coefficients and all its roots have modulus one, so all its roots are roots of unity by the theorem of Kronecker \cite{Gr}. The eigenvalues of $B$ are eigenvalues of $A$, and this contradicts ergodicity, completing the proof.
\end{proof}

\begin{remark}
When $\dim E^c = 2$ there is a short geometric alternative. The closure $\Gamma$ is then a closed subgroup of the plane invariant under $A^c$, which is a rotation by an irrational angle by Lemma \ref{lem:linearalgebra}. If $\Gamma$ were trivial then $\Z^N$ would be contained in $E^s \oplus E^u$, if $\Gamma$ were discrete and nonzero it would contain an infinite bounded orbit of the rotation, and if the connected component of the identity of $\Gamma$ were a line then that line would be invariant under the rotation. These are the only possible proper cases and each case is impossible, so $\Gamma$ is the whole plane. This argument does not extend to center spaces of dimension larger than two, where $A^c$ may admit proper invariant subspaces, while the proof given above does.
\end{remark}

\begin{lemma}\label{lem:fix}
Let $A \in \GL(N,\Z)$ have no eigenvalue equal to one. Then $\Fix(A)$ is isomorphic to $\Z^N / (A - I)\Z^N$ and has order $|p_A(1)|$.
\end{lemma}

\begin{proof}
A point $w = p(w_0)$ lies in $\Fix(A)$ if and only if $(A - I) w_0 \in \Z^N$, that is if and only if $w_0 \in (A - I)^{-1} \Z^N$. The map $w_0 \mapsto (A - I) w_0$ induces an isomorphism between $(A-I)^{-1}\Z^N / \Z^N$ and $\Z^N / (A - I)\Z^N$, and the order of the latter group is $|\det(A - I)| = |p_A(1)|$.
\end{proof}

\begin{lemma}\label{lem:commutant}
Let $A \in \GL(N,\Z)$ have irreducible characteristic polynomial. Then the centralizer of $A$ in the algebra of rational matrices equals $\Q[A]$, a field isomorphic to $K = \Q(\lambda)$ under the map sending $A$ to $\lambda$. The set $\mathcal{O}_A = \Q[A] \cap M_N(\Z)$ is an order in $K$, and the centralizer of $A$ in $\GL(N,\Z)$ equals the unit group $\mathcal{O}_A^{\times}$. This group is a finitely generated abelian group whose free rank equals $r_1 + r_2 - 1$.
\end{lemma}

\begin{proof}
Since $p_A$ is irreducible it coincides with the minimal polynomial of $A$, so $A$ is conjugate over $\Q$ to the companion matrix of $p_A$ and admits a cyclic vector. The centralizer of a matrix with a cyclic vector consists of the polynomials in that matrix, so the centralizer of $A$ in $M_N(\Q)$ is $\Q[A] \cong \Q[t]/(p_A) \cong K$, which is a field by irreducibility. The set $\mathcal{O}_A$ is a subring of $\Q[A]$ containing the identity. It is finitely generated as an abelian group because it is contained in $M_N(\Z)$, and it contains $\Z[A]$, hence it has rank $N$ over $\Z$ and is an order in $K$. If $B \in \GL(N,\Z)$ commutes with $A$ then $B \in \mathcal{O}_A$, and $B^{-1}$ commutes with $A$ as well and is integral, so $B^{-1} \in \mathcal{O}_A$ and $B \in \mathcal{O}_A^{\times}$. The converse inclusion is clear. The final assertion is the Dirichlet unit theorem for orders in number fields, see \cite{BS}.
\end{proof}

\begin{lemma}\label{lem:quartic}
Let $A$ be as in Lemma \ref{lem:linearalgebra}, let $\lambda$ be the eigenvalue of modulus one with positive imaginary part, and let $q \in \Z[t]$ be its minimal polynomial. Then $q$ divides $p_A$, the set of roots of $q$ is stable under $\mu \mapsto \mu^{-1}$, the degree of $q$ is even and at least four, and $q$ has roots of modulus different from one. In particular, if $N = 4$ then $q = p_A$ and $p_A$ is irreducible.
\end{lemma}

\begin{proof}
The number $\lambda$ is an algebraic integer, so its minimal polynomial $q$ is monic with integer coefficients and divides $p_A$. Since $|\lambda| = 1$ we have $\lambda^{-1} = \bar\lambda$, which is a root of $q$ because $q$ has real coefficients. Every root of $q$ has the form $\sigma(\lambda)$ for an embedding $\sigma$ of $\Q(\lambda)$ into $\C$, and $\sigma(\lambda)^{-1} = \sigma(\lambda^{-1})$ is then the image under $\sigma$ of the root $\lambda^{-1} \in \Q(\lambda)$ of $q$, hence again a root of $q$. This proves stability under inversion. A root fixed by inversion would equal $\pm 1$, which is impossible for an irreducible polynomial with the nonreal root $\lambda$, so the roots of $q$ split into pairs $\{\mu, \mu^{-1}\}$ with $\mu \neq \mu^{-1}$ and the degree of $q$ is even. If all roots of $q$ had modulus one they would be roots of unity by the theorem of Kronecker \cite{Gr}, contradicting the choice of $\lambda$, so $q$ has roots of modulus different from one, and in particular the degree of $q$ is not two, since a monic reciprocal integer quadratic with a nonreal root has both roots on the unit circle, hence roots of unity. Therefore the degree of $q$ is at least four. For $N = 4$ this forces $q = p_A$. Compare \cite[Lemmas 13 and 14]{AU}.
\end{proof}

Now we recall a definition of accessibility property. For a partially hyperbolic diffeomorphism $f$ the accessibility class $C(x)$ of a point $x$ is the set of points that can be joined to $x$ by a path consisting of finitely many arcs each of which is contained in a single strong stable or strong unstable leaf of $f$, and $f$ has the accessibility property if $C(x) = \T^N$ for some, equivalently every, point $x$. \\

The following statement combines Theorem 1 of \cite{AU} with the structure of the non accessible case established in Section 6 of that paper, which relies on \cite{RV} for the submanifold structure of accessibility classes and on Sections 3 to 6 of \cite{RH}.

\begin{theorem}\label{thm:au}
Let $A \in \SL(N,\Z)$ induce an ergodic automorphism of $\T^N$ with $\dim E^c = 2$. There is a neighborhood $\mathcal{U}$ of $A$ in the space of volume preserving $C^{22}$ diffeomorphisms of $\T^N$ such that every $f \in \mathcal{U}$ is ergodic with respect to Lebesgue measure and such that for every $f \in \mathcal{U}$ either $f$ has the accessibility property, or the strong stable and strong unstable bundles of $f$ are jointly integrable and the accessibility classes of $f$ are the leaves of an $f$ invariant topological foliation whose leaves are closed sets of codimension two.
\end{theorem}

In the second case of the alternative further structure is available from \cite[Section 6]{AU} and \cite[Sections 3 to 6]{RH}. We list it at the beginning of Section 4, where it is used.

\begin{remark}
When $A$ is pseudo-Anosov in the sense of \cite{RH} and $N \geq 6$, all the statements of this paper hold with $\mathcal{U}$ replaced by the $C^5$ neighborhood of \cite[Theorem 1.1]{RH}, since the alternative and the structure of the jointly integrable case are established there in that regularity.
\end{remark}

\begin{remark}\label{rem:dimensions}
The class covered by Theorem \ref{thm:au} is large. By \cite{AU} every ergodic non Anosov automorphism of $\T^N$ has two dimensional center when $N \leq 5$ or $N = 7$, and the same holds in dimensions $6$ and $9$ provided no eigenvalue of $A$ is a Salem number. Consequently, in dimensions $N \leq 5$ and $N = 7$ the dichotomy of Theorem \ref{thm:main} applies to every ergodic automorphism, since in the Anosov case the conjugacy is provided by structural stability for every $C^1$ small perturbation and the centralizer description then follows from the theorem of Adler and Palais alone. A concrete family of examples in dimension four is given by the companion matrices of Salem polynomials of degree four, for instance $t^4 - t^3 - t^2 - t + 1$, whose roots consist of a pair of real numbers $\mu$ and $\mu^{-1}$ with $\mu > 1$ together with a pair of nonreal numbers of modulus one which are not roots of unity.
\end{remark}

\section{Homeomorphisms commuting with an ergodic automorphism}

Thrughout this section $A \in \GL(N,\Z)$ is ergodic and the restriction $A^c$ of $A$ to its center space is semisimple. By Lemma \ref{lem:linearalgebra} this holds for every ergodic automorphism with two dimensional center, and it holds vacuously when $A$ is hyperbolic. For automorphisms the affine rigidity below is due to Adler and Palais \cite{AP}, whose hypothesis is ergodicity, and Walters \cite{W1} proved that ergodicity is also necessary, later extending the theory to affine transformations of compact abelian groups \cite{W2}. We include complete proofs because the following lemma is reused in Section 5 and because the precise semisimple center formulation is the one needed later.

\begin{lemma}\label{lem:transfer}
Let $\varphi \colon \T^N \to \R^N$ be a continuous map satisfying
\[
\varphi(A x) = A \, \varphi(x) \qquad \text{for every } x \in \T^N .
\]
Then $\varphi$ vanishes identically.
\end{lemma}

\begin{proof}
Decompose $\varphi = \varphi^s + \varphi^c + \varphi^u$ according to the splitting $\R^N = E^s \oplus E^c \oplus E^u$. The splitting is $A$ invariant, so
\[
\varphi^\sigma(A x) = A^\sigma \varphi^\sigma(x), \qquad \sigma \in \{s, c, u\} .
\]
Iterating the relation for $\sigma = s$ gives $\varphi^s(x) = (A^s)^n \varphi^s(A^{-n} x)$ for every $n \geq 1$. Since the spectral radius of $A^s$ is smaller than one there are $C > 0$ and $\theta \in (0,1)$ with $\| (A^s)^n \| \leq C \theta^n$, and since $\varphi^s$ is continuous on a compact space it is bounded. Letting $n \to \infty$ yields $\varphi^s = 0$. In the same way $\varphi^u(x) = (A^u)^{-n} \varphi^u(A^n x)$ and $\| (A^u)^{-n} \| \leq C \theta^n$, so $\varphi^u = 0$.

It remains to treat the center component. Since $A^c$ is semisimple with eigenvalues of modulus one, choosing a basis of $E^c \otimes_\R \C$ consisting of eigenvectors of the complexification of $A^c$ and declaring it orthonormal produces a Hermitian norm $\| \cdot \|$ with respect to which the complexification of $A^c$ is an isometry. Write the Fourier expansion
\[
\varphi^c(x) = \sum_{m \in \Z^N} a_m \, e^{2\pi i \langle m, x \rangle}, \qquad a_m \in E^c \otimes_\R \C ,
\]
where the coefficients are computed with respect to Lebesgue measure. Then
\[
\varphi^c(A x) = \sum_{m \in \Z^N} a_m \, e^{2\pi i \langle A^T m, x \rangle} = \sum_{k \in \Z^N} a_{(A^T)^{-1} k} \, e^{2\pi i \langle k, x \rangle},
\]
so the equivariance relation gives $a_{(A^T)^{-1} k} = A^c \, a_k$ for every $k \in \Z^N$, equivalently
\[
a_k = A^c \, a_{A^T k}, \qquad k \in \Z^N .
\]
Hence $\| a_k \| = \| a_{A^T k} \|$, and the function $k \mapsto \| a_k \|$ is constant along the orbits of $A^T$ on $\Z^N$. For $k \neq 0$ these orbits are infinite and unbounded, since no eigenvalue of $A$ is a root of unity, while the coefficients of the continuous, hence integrable, map $\varphi^c$ tend to zero as $|k| \to \infty$ by the Riemann - Lebesgee lemma. Therefore $a_k = 0$ for every $k \neq 0$. For $k = 0$ the relation reads $a_0 = A^c a_0$, and since one is not an eigenvalue of $A$ we get $a_0 = 0$. Hence $\varphi^c = 0$, which completes the proof.
\end{proof}

\begin{proposition}\label{prop:affine}
Let $g$ be a homeomorphism of $\T^N$ with $A g = g A$. Then $g$ is affine. More precisely $g(x) = B x + w$ where $B \in \GL(N,\Z)$ is the automorphism induced by $g$ on the first homology group, $A B = B A$, and $w \in \Fix(A)$.
\end{proposition}

\begin{proof}
Let $B \in \GL(N,\Z)$ be the matrix of the action of $g$ on $H_1(\T^N, \Z) \cong \Z^N$ and let $G \colon \R^N \to \R^N$ be a lift of $g$, so that $p \circ G = g \circ p$ and
\[
G(x + n) = G(x) + B n, \qquad n \in \Z^N .
\]
Functoriality of homology applied to the relation $g A = A g$ gives $B A = A B$. The maps $G \circ A$ and $A \circ G$ are both lifts of the map $g A = A g$ of the torus, so their difference is a continuous map with values in $\Z^N$, hence a constant, that is
\[
G(A x) = A \, G(x) + m_0 \qquad \text{for some } m_0 \in \Z^N \text{ and every } x \in \R^N .
\]
Set
\[
k = B^{-1} \circ G + v, \qquad v = (A - I)^{-1} B^{-1} m_0 ,
\]
which is well defined by ergodicity. Then $k(x + n) = k(x) + n$ for every $n \in \Z^N$, and using $A B^{-1} = B^{-1} A$ we compute

\begin{align*}
k \big(A x \big) &= B^{-1} \big( G(A \, x) \big) + v = B^{-1} \big( A \, G(x) + m_0 \big) + v \\ &= A \big( k(x) - v \big) + B^{-1} m_0 + v = A \, k(x) + (I - A) v + B^{-1} m_0 \\ &= A \, k(x),
\end{align*}

by the choice of $v$. Therefore $\varphi = k - \mathrm{id}$ is continuous, $\Z^N$ periodic, and satisfies $\varphi(A x) = A \varphi(x)$, so it descends to a map on the torus to which Lemma \ref{lem:transfer} applies. Hence $\varphi = 0$, that is $G(x) = B(x - v) = B x + w_0$ with $w_0 = - B v$.

It remains to check that $w = p(w_0)$ lies in $\Fix(A)$. Using $B(A - I) = (A - I) B$ we get
\[
A w_0 - w_0 = -(A - I) B v = - B (A - I) v = - B B^{-1} m_0 = - m_0 \in \Z^N ,
\]
so $A w = w$ on the torus. Finally $g(x) = B x + w$ since $G$ is a lift of $g$.
\end{proof}

\begin{remark}
Invertibility of $g$ was used only to conclude that $B \in \GL(N,\Z)$. The same proof shows that every continuous map of $\T^N$ commuting with $A$ whose induced matrix on the first homology group is nonsingular is affine.
\end{remark}

\begin{corollary}\label{cor:centralizerA}
Let $A$ be as above. Then the centralizer of $A$ in $\Homeo(\T^N)$ equals $\Aff_A$. The map sending $x \mapsto B x + w$ to $B$ is a surjective homomorphism from $\Aff_A$ onto the centralizer of $A$ in $\GL(N,\Z)$ with kernel the group of translations by elements of $\Fix(A)$, and this extension splits, so
\[
\Aff_A \cong \Fix(A) \rtimes C_{\GL(N,\Z)}(A) .
\]
If moreover $p_A$ is irreducible then $C_{\GL(N,\Z)}(A) = \mathcal{O}_A^{\times}$ and $\Aff_A$ is virtually free abelian of rank $r_1 + r_2 - 1$.
\end{corollary}

\begin{proof}
An affine map $x \mapsto B x + w$ with $A B = B A$ and $A w = w$ commutes with $A$, since $A(Bx + w) = BAx + Aw = B(Ax) + w$. Conversely every homeomorphism commuting with $A$ has this form by Proposition \ref{prop:affine}. The homomorphism onto the linear parts is surjective because $w = 0$ is allowed, its kernel consists of the translations by $\Fix(A)$, and the assignment $B \mapsto (x \mapsto Bx)$ is a homomorphic section. The last assertion follows from Lemma \ref{lem:commutant}, since the subgroup of linear elements of $\Aff_A$ is then abelian of free rank $r_1 + r_2 - 1$ and has finite index $|\Fix(A)|$.
\end{proof}

\section{The conjugacy in the jointly integrable case}

In this section we prove the following proposition, which extends the conjugacy statement of \cite[Section 6]{RH} from the pseudo-Anosov class to the full two dimensional center class of \cite{AU}.

\begin{proposition}\label{prop:conjugacy}
Let $f \in \mathcal{U}$ and suppose that $f$ does not have the accessibility property. Then there is a homeomorphism $H$ of $\T^N$ homotopic to the identity such that $H \circ f = A \circ H$.
\end{proposition}

Since $f$ is $C^0$ close to $A$, it is homotopic to $A$, so the lift $F$ of $f$ with $F - A$ bounded satisfies
\[
F(x + n) = F(x) + A n, \qquad n \in \Z^N ,
\]
and $\psi := F - A$ is continuous, $\Z^N$ periodic and bounded. We denote by $\widetilde{W}^s(x)$ and $\widetilde{W}^u(x)$ the lifted strong stable and strong unstable leaves of $f$ through $x \in \R^N$ and by $\widetilde{C}(x)$ the lifted accessibility class, which in the jointly integrable case is the leaf through $x$ of the lifted joint foliation. We make use the following three facts, all established in \cite[Section 6]{AU} and \cite[Sections 3 to 6]{RH} for maps in $\mathcal{U}$ in the jointly integrable case.

\begin{enumerate}[label=Fact \arabic*:]
    \item Fibration structure. The center leaf $W^c(0)$ of $f$ through the origin of $\R^N$ is a two dimensional topological manifold which is a graph over $E^c$, so that the projection $\pi^c$ restricts to a homeomorphism of $W^c(0)$ onto $E^c$, and there is a continuous surjection $\pi^{su} \colon \R^N \to W^c(0)$ whose fibers are exactly the lifted accessibility classes. Identifying $W^c(0)$ with $E^c$ through $\pi^c$, we consider the composition $\bar\pi = \pi^c \circ \pi^{su}$ as a continuous surjection of $\R^N$ onto $E^c$ whose fibers are the lifted accessibility classes. For every $n \in \Z^N$ the formula
\[
\bar\pi(x + n) = T_n\big( \bar\pi(x) \big)
\]
defines a homeomorphism $T_n$ of $E^c$, and in the jointly integrable case these maps satisfy $T_{n + m} = T_n \circ T_m$ for all $n, m \in \Z^N$, so that $n \mapsto T_n$ is an action of $\Z^N$ on $E^c$.

\item Linearization of this action. By the construction of \cite[Section 6]{RH}, carried out in the present setting in \cite[Section 6]{AU} through the Diophantine estimates of \cite[Lemmas 17 and 18]{AU} and the theorems of Herman and Moser, there is a homeomorphism $h$ of $E^c$ such that
\[
h \circ T_n \circ h^{-1} = R_{n^c} \qquad \text{for every } n \in \Z^N ,
\]
where $R_v$ denotes the translation of $E^c$ by the vector $v$.

\item Graph and product structure of the strong leaves. Each lifted strong stable leaf is a Lipschitz graph over a translate of $E^s$ with slope bounded uniformly over all leaves, so that $\pi^s$ restricts to a homeomorphism of each leaf onto $E^s$, and there are constants $c > 0$ and $\sigma > 1$ such that
\[
\big| \pi^s\big(F^{-k} y\big) - \pi^s\big(F^{-k} z\big) \big| \geq c \, \sigma^k \, \big| \pi^s(y) - \pi^s(z) \big|
\]
for all $k \geq 0$ and all $y, z$ on the same lifted strong stable leaf. The symmetric statements hold for the strong unstable leaves under forward iteration. Moreover, in the jointly integrable case each lifted leaf of the joint foliation carries a global product structure, meaning that for $x$ and $y$ on the same leaf the intersection $\widetilde{W}^s(x) \cap \widetilde{W}^u(y)$ consists of exactly one point.

\end{enumerate}

\begin{proof}[Proof of Proposition \ref{prop:conjugacy}]
We divide the construction into five steps.

\begin{enumerate}[label=Step \arabic*:]
    \item We fist establish the quivariance of the fibration. The map $F$ sends lifted strong stable leaves of $f$ to lifted strong stable leaves and likewise for the unstable ones, so it sends lifted accessibility classes to lifted accessibility classes, that is $F(\widetilde{C}(x)) = \widetilde{C}(F(x))$. Consequently there is a well defined continuous map $P$ of $E^c$ with
\[
\bar\pi \circ F = P \circ \bar\pi ,
\]
namely $P(v) = \bar\pi(F(x))$ for any $x$ with $\bar\pi(x) = v$, and $P$ is a homeomorphism with inverse induced by $F^{-1}$ in the same way. From $F(x + n) = F(x) + An$ we obtain, for every $n \in \Z^N$ and every $x$,
\[
P\big( T_n (\bar\pi(x)) \big) = \bar\pi\big( F(x + n) \big) = \bar\pi\big( F(x) + An \big) = T_{An}\big( P(\bar\pi(x)) \big),
\]
and since $\bar\pi$ is surjective this gives $P \circ T_n = T_{An} \circ P$.

\item We normalize the induced center dynamics. Set $Q = h \circ P \circ h^{-1}$, a homeomorphism of $E^c$. Conjugating the relation of the first step by $h$ and using the second fact gives $Q \circ R_{n^c} = R_{(An)^c} \circ Q$, and $(An)^c = A^c n^c$ by the invariance of the splitting, so
\[
Q(v + n^c) = Q(v) + A^c n^c \qquad \text{for every } v \in E^c \text{ and } n \in \Z^N .
\]
The map $\psi_0(v) = Q(v) - A^c v$ is then continuous and satisfies $\psi_0(v + n^c) = \psi_0(v)$ for every $n$, hence it is invariant under a dense group of translations by Lemma \ref{lem:density} and is therefore constant, say $\psi_0 \equiv \beta$. Replacing $h$ by $R_w \circ h$ with $w = (A^c - I)^{-1} \beta$, which exists because one is not an eigenvalue of $A^c$ and which does not affect the second fact since translations commute with translations, replaces $Q$ by $R_w \circ Q \circ R_{-w}$, whose value at $v$ is $A^c v + \beta + (I - A^c) w = A^c v$. After this replacement we therefore have
\[
h \circ T_n \circ h^{-1} = R_{n^c} \quad \text{and} \quad h \circ P \circ h^{-1} = A^c .
\]

\item We establish the center component of the conjugacy. Define $h^c = h \circ \bar\pi \colon \R^N \to E^c$. Then $h^c$ is continuous and surjective, its fibers are the lifted accessibility classes, and by the first two steps
\[
h^c(x + n) = h^c(x) + n^c \quad \text{and} \quad h^c \circ F = A^c \circ h^c .
\]
In particular $h^c(x) - x^c$ is $\Z^N$ periodic and continuous, hence bounded.

\item We establish the hyperbolic components by solving cohomological equations. Let $\psi^s = \psi^s(x)$ and $\psi^u = \psi^u(x)$ denote the $E^s$ and $E^u$ components of $\psi(x) = F(x) - Ax$, which are continuous, $\Z^N$ periodic and bounded. Define
\[
\varphi^s(x) = - \sum_{j \geq 1} (A^s)^{j-1} \, \psi^s\big( F^{-j} x \big), \qquad \varphi^u(x) = \sum_{j \geq 0} (A^u)^{-j-1} \, \psi^u\big( F^{j} x \big) .
\]
Both series converge uniformly, since $\| (A^s)^{j} \| \leq C \theta^{j}$ and $\| (A^u)^{-j} \| \leq C \theta^{j}$ with $\theta \in (0,1)$ and the maps $\psi^s$, $\psi^u$ are bounded, so $\varphi^s$ and $\varphi^u$ are continuous and bounded. They are $\Z^N$ periodic because $F^{\pm j}(x + n) = F^{\pm j}(x) + A^{\pm j} n$ with $A^{\pm j} n \in \Z^N$ and $\psi^s$, $\psi^u$ are periodic. A direct computation gives
\[
\varphi^s(F x) = - \sum_{i \geq 0} (A^s)^{i} \, \psi^s\big( F^{-i} x \big) = A^s \varphi^s(x) - \psi^s(x),
\]
and similarly $\varphi^u(F x) = A^u \varphi^u(x) - \psi^u(x)$. Setting $h^s = \pi^s + \varphi^s$ and $h^u = \pi^u + \varphi^u$ we obtain, for every $x$ and $n$,
\begin{align*}
&h^s \circ F = A^s \circ h^s, \qquad h^u \circ F = A^u \circ h^u, \\ &h^s(x+n) = h^s(x) + n^s, \qquad h^u(x+n) = h^u(x) + n^u.
\end{align*}

\item We assemble the conjugacy. Define $H_1 = h^s + h^u + h^c \colon \R^N \to \R^N$, using the splitting $\R^N = E^s \oplus E^u \oplus E^c$ to interpret the sum. By the previous steps $H_1 \circ F = A \circ H_1$ and $H_1(x + n) = H_1(x) + n$ for every $n \in \Z^N$, and $H_1 - \mathrm{id} = \varphi^s + \varphi^u + (h^c - \pi^c)$ is continuous, $\Z^N$ periodic and bounded.

We claim that $H_1$ is injective. Note first that $h^s$ is uniformly continuous, being the sum of a linear map and a periodic continuous map, and constant on lifted strong unstable leaves. Indeed, if $y \in \widetilde{W}^u(x)$ then $d(F^{-k}x, F^{-k}y) \to 0$, while
\[
h^s\big(F^{-k}x\big) - h^s\big(F^{-k}y\big) = (A^s)^{-k}\big( h^s(x) - h^s(y) \big),
\]
whose norm is at least $C^{-1} \theta^{-k} \, | h^s(x) - h^s(y) |$ and hence tends to infinity unless $h^s(x) = h^s(y)$. Uniform continuity forces the left hand side to tend to zero, so $h^s(x) = h^s(y)$. Next, $h^s$ is injective on each lifted strong stable leaf. If $y \neq z$ lie on the same lifted strong stable leaf then $\pi^s(y) \neq \pi^s(z)$ by the graph structure, and if $h^s(y) = h^s(z)$ then $h^s(F^{-k}y) = (A^s)^{-k} h^s(y) = h^s(F^{-k}z)$ for all $k$, so
\[
\pi^s\big(F^{-k}y\big) - \pi^s\big(F^{-k}z\big) = \varphi^s\big(F^{-k}z\big) - \varphi^s\big(F^{-k}y\big)
\]
is bounded, contradicting the backward expansion of the third fact. The symmetric statements hold for $h^u$, which is constant on lifted strong stable leaves and injective on each lifted strong unstable leaf. Now suppose $H_1(x) = H_1(y)$. Comparing center components gives $h^c(x) = h^c(y)$, so $x$ and $y$ lie on the same lifted accessibility class, and by the global product structure there is a unique point $z \in \widetilde{W}^s(x) \cap \widetilde{W}^u(y)$. Then $h^s(z) = h^s(y) = h^s(x)$, the first equality because $z$ and $y$ lie on the same strong unstable leaf and the second by assumption, so $z = x$ by injectivity of $h^s$ on the strong stable leaf through $x$. Hence $y \in \widetilde{W}^u(x)$, and $h^u(y) = h^u(x)$ forces $y = x$ by injectivity of $h^u$ on that leaf. This proves the claim.

Since $H_1$ is continuous and injective it is an open map by invariance of domain, and since $H_1 - \mathrm{id}$ is bounded the map $H_1$ is proper, hence closed. Its image is therefore open, closed and nonempty, so $H_1$ is a bijection of $\R^N$, and an open continuous bijection is a homeomorphism. Since $H_1(x+n) = H_1(x) + n$, the map $H_1$ descends to a homeomorphism $H$ of $\T^N$, homotopic to the identity through the maps $\mathrm{id} + t(H_1 - \mathrm{id})$, and the relation $H_1 \circ F = A \circ H_1$ descends to $H \circ f = A \circ H$.
\end{enumerate}
\end{proof}

\begin{remark}\label{rem:regularity}
The conjugacy $H$ is in general only continuous. This is precisely the regularity relevant for the applications below, since the affine rigidity of Section 3 is purely topological. The $C^1$ regularity available in the jointly integrable case, coming from the theorems of Herman and Moser, attaches to the conjugacy between the strong foliations of $f$ and the linear foliations of $A$. That is, to the map $x \mapsto x^s + x^u + h^c(x)$ considered in \cite{RH} and \cite{AU}, and not to the dynamical conjugacy $H$ itself.
\end{remark}

\section{Proofs of the main results}

We begin with the proposition which gives the exclusivity of the alternative in Theorem \ref{thm:main}. For volume preserving smooth perturbations sufficiently close to the automorphism the equivalence between the failure of accessibility and topological conjugacy was observed in \cite{Sa} building on \cite{RH}. The version below requires neither closeness of $f$ to $A$ nor volume preservation, the conjugating homeomorphism is not assumed to be homotopic to the identity, and the proof is elementary and self contained. The argument combines \emph{the metric characterization of the stable and unstable sets of the linear model} with the countability of the projection of the lattice to the center space, which is a weak form of the property expressed by Lemma \ref{lem:density}.

\begin{proposition}\label{prop:exclusive}
Let $A \in \GL(N,\Z)$ induce an ergodic automorphism of $\T^N$ with $E^c \neq \{0\}$ and with semisimple center restriction $A^c$. Let $f$ be a partially hyperbolic diffeomorphism of $\T^N$ and let $H'$ be a homeomorphism of $\T^N$ with $H' \circ f = A \circ H'$. Then $H'$ maps every accessibility class of $f$ into a coset of $p(E^s \oplus E^u)$, which is a proper subgroup of $\T^N$. In particular $f$ does not have the accessibility property.
\end{proposition}

\begin{proof}
We first describe the metric stable set of the linear model. We claim that for $w \in \T^N$ one has $d(A^n w, 0) \to 0$ as $n \to +\infty$ if and only if $w \in p(E^s)$, and $d(A^n w, 0) \to 0$ as $n \to -\infty$ if and only if $w \in p(E^u)$. One implication is cleear, since for $x \in E^s$ the orbit $A^n x$ converges to the origin in $\R^N$ and the projection $p$ does not increase distances. For the converse, suppose $d(A^n w, 0) \to 0$ and choose for every $n$ a lift $x_n$ of $A^n w$ with $|x_n| = d(A^n w, 0)$. The vector $A x_n$ is a lift of $A^{n+1} w$, so $m_n = x_{n+1} - A x_n$ belongs to $\Z^N$ and satisfies $|m_n| \leq |x_{n+1}| + \|A\| \, |x_n| \to 0$, hence $m_n = 0$ for all $n \geq n_0$. Therefore $x_n = A^{\,n - n_0} x_{n_0}$ for $n \geq n_0$ and $x_n \to 0$. Write $x_{n_0} = v^s + v^c + v^u$. The unstable component of $x_n$ equals $(A^u)^{\,n - n_0} v^u$ and satisfies $|(A^u)^{k} v^u| \geq C^{-1} \theta^{-k} |v^u|$ with $\theta \in (0,1)$, which is compatible with $x_n \to 0$ only if $v^u = 0$. The center component of $x_n$ equals $(A^c)^{\,n - n_0} v^c$ and has constant norm with respect to a norm invariant under $A^c$, which exists by semisimplicity as in the proof of Lemma \ref{lem:transfer}, so $v^c = 0$ as well. Hence $x_{n_0} \in E^s$, so $A^{n_0} w = p(x_{n_0}) \in p(E^s)$, and since $E^s$ is invariant under $A^{-1}$ we conclude $w = p(A^{-n_0} x_{n_0}) \in p(E^s)$. The statement for $n \to -\infty$ follows by applying the same argument to $A^{-1}$, whose center restriction is again semisimple with eigenvalues of modulus one. Since the metric on $\T^N$ is invariant under translations and $A^n x - A^n y = A^n (x - y)$, it follows that $d(A^n x, A^n y) \to 0$ as $n \to +\infty$ if and only if $y - x \in p(E^s)$, and similarly for $n \to -\infty$ with $p(E^u)$.

Now let $x \in \T^N$ and let $y$ belong to the strong stable leaf of $f$ through $x$. Then $d(f^n x, f^n y) \to 0$ as $n \to +\infty$, and since $H'$ is uniformly continuous on the compact torus,
\[
d\big( A^n H'(x), \, A^n H'(y) \big) = d\big( H'(f^n x), \, H'(f^n y) \big) \longrightarrow 0 ,
\]
so $H'(y) - H'(x) \in p(E^s)$ by the first part of the proof. In the same way, if $y$ belongs to the strong unstable leaf of $f$ through $x$ then $H'(y) - H'(x) \in p(E^u)$. If $y$ lies in the accessibility class of $x$ there is a chain $x = z_0, z_1, \dots, z_k = y$ in which each $z_{i+1}$ belongs to the strong stable or strong unstable leaf of $f$ through $z_i$, and summing the corresponding differences gives
\[
H'(y) - H'(x) = \sum_{i=0}^{k-1} \big( H'(z_{i+1}) - H'(z_i) \big) \in p(E^s) + p(E^u) = p(E^s \oplus E^u) ,
\]
since $p$ is a group homomorphism. Hence $H'$ maps the accessibility class of $x$ into $$H'(x) + p(E^s \oplus E^u).$$

It remains to check that $p(E^s \oplus E^u)$ is a proper subgroup of $\T^N$. It is a subgroup, being the image of a subspace under the projection homomorphism. If it were the whole torus we would have $\R^N = (E^s \oplus E^u) + \Z^N$, and applying the projection onto $E^c$ along $E^s \oplus E^u$ would give $E^c = \{ n^c \, \colon \, n \in \Z^N \}$, exhibiting the uncountable set $E^c$ as countable, a contradiction. Finally, if $f$ had the accessibility property, the accessibility class of any point would be the whole torus, and surjectivity of $H'$ would force $\T^N = H'(x) + p(E^s \oplus E^u)$, contradicting properness of the subgroup.
\end{proof}

\begin{remark}
Applied with $f = A$ and $H'=Id$, Proposition \ref{prop:exclusive} shows that no ergodic automorphism with nontrivial center space and semisimple center restriction has the accessibility property. Its accessibility classes being contained in the cosets of $p(E^s \oplus E^u)$. Accessibility is not in general invariant under topological conjugacy, since it is defined through the strong foliations, and the content of the proposition is that conjugacy to the linear model nevertheless forces the strong foliations of $f$ into the metrically defined stable and unstable sets of $A$.
\end{remark}

\begin{proof}[Proof of Theorem \ref{thm:main}]
Let $f \in \mathcal{U}$ and let $f$ does not have the accessibility property. By Proposition \ref{prop:conjugacy} there is a homeomorphism $H$ of $\T^N$ homotopic to the identity with $H \circ f \circ H^{-1} = A$. Let $g$ be any homeomorphism of $\T^N$ commuting with $f$ and set $\hat g = H \circ g \circ H^{-1}$. Then
\[
\hat g \, A = H \circ g \circ H^{-1} \circ H \circ f \circ H^{-1} = H \circ g \circ f \circ H^{-1} = H \circ f \circ g \circ H^{-1} = A \, \hat g ,
\]
so $\hat g$ is a homeomorphism commuting with $A$, and $A$ satisfies the hypotheses of Section 3 by Lemma \ref{lem:linearalgebra}. By Proposition \ref{prop:affine} the map $\hat g$ is affine, $\hat g(x) = B x + w$ with $B$ the matrix induced by $\hat g$ on the first homology group and $w \in \Fix(A)$. Since $H$ is homotopic to the identity it acts trivially on homology, so $B$ is also the matrix induced by $g$ itself. Conjugation by $H$ therefore maps $Z(f)$ into $\Aff_A$, and it maps onto $\Aff_A$ because every element of $\Aff_A$ commutes with $A$ by Corollary \ref{cor:centralizerA}, so its conjugate by $H^{-1}$ commutes with $f$. Being a conjugation, this map is a group isomorphism. This proves that at least one of the two assertions holds. They cannot hold simultaneously, since the second provides a topological conjugacy between $f$ and $A$, and Proposition \ref{prop:exclusive} then shows that $f$ does not have the accessibility property.
\end{proof}

\begin{proof}[Proof of Corollary \ref{cor:equivalence}]
First suppose the strong stable and strong unstable bundles of $f$ are jointly integrable. Every strong stable and every strong unstable leaf of $f$ is contained in a leaf of the foliation integrating $E^s_f \oplus E^u_f$, so every accessibility class of $f$ is contained in a single leaf. A leaf is an injectively immersed submanifold of dimension $N - 2$, hence a set of Lebesgue measure zero and in particular a proper subset of the torus, so $f$ is not accessible. Conversely, if $f$ is not accessible then the bundles are jointly integrable by Theorem \ref{thm:au}. Moreover, if $f$ is not accessible then $f$ is topologically conjugate to $A$ by Proposition \ref{prop:conjugacy}. Finally, if $f$ is topologically conjugate to $A$ then $f$ is not accessible by Proposition \ref{prop:exclusive}. This settles the equivalence of the three statements.
\end{proof}

\begin{remark}\label{rem:uniqueness}
The homeomorphism $H$ in the second alternative is unique up to composition with a translation by an element of $\Fix(A)$. Indeed, if $H'$ is another homeomorphism homotopic to the identity with $H' \circ f \circ H'^{-1} = A$, then $H' \circ H^{-1}$ commutes with $A$, is homotopic to the identity, and is therefore an affine map with linear part equal to the identity by Proposition \ref{prop:affine}, that is a translation by an element of $\Fix(A)$. In particular the isomorphism of Theorem \ref{thm:main} does not depend on the choice of $H$ up to an inner automorphism of $\Aff_A$.
\end{remark}

\begin{proof}[Proof of Corollary \ref{cor:structure}]
The first assertion follows at once from Theorem \ref{thm:main} and Corollary \ref{cor:centralizerA}, and the irreducible case is the last assertion of Corollary \ref{cor:centralizerA}. Suppose $N = 4$. By Lemma \ref{lem:quartic} the polynomial $p_A$ is irreducible, with roots $\lambda$, $\bar\lambda$ of modulus one and two further roots $\mu_1$, $\mu_2$. If $\mu_1$ and $\mu_2$ were nonreal they would form a conjugate pair with $|\mu_1| = |\mu_2| \neq 1$, while $|\det A| = 1$ forces $|\lambda|^2 |\mu_1 \mu_2| = 1$ and hence $|\mu_1|^2 = 1$, a contradiction. Therefore $\mu_1$ and $\mu_2$ are real, so $K$ has $r_1 = 2$ real embeddings and $r_2 = 1$ pair of complex embeddings, and the rank equals $r_1 + r_2 - 1 = 2$. The last assertion follows because $Z(f)$ contains $f$ together with the conjugates by $H^{-1}$ of the linear maps in $\mathcal{O}_A^{\times}$, which form a free abelian group of rank two up to finite index.
\end{proof}

\begin{proof}[Proof of Corollary \ref{cor:criterion}]
We prove it by contraposition. Assume that $f$ does not have the accessibility property. By Theorem \ref{thm:main} the group $Z(f)$ is isimorphic to $\Aff_A$. If $g \in Z(f)$ acts trivially on homology then $H \circ g \circ H^{-1}$ is a translation by an element of $\Fix(A)$ and hence has finite order, so $g$ has finite order. Moreover $\Aff_A$ is countable, and when $p_A$ is irreducible it is virtually abelian by Corollary \ref{cor:centralizerA}.
\end{proof}

\begin{proof}[Proof of Corollary \ref{cor:dynamics}]
By Theorem \ref{thm:main} we have $H \circ g \circ H^{-1}(x) = B x + w$ with $w \in \Fix(A)$. Since no eigenvalue of $B$ is a root of unity, one is not an eigenvalue of $B$, so $\det(B - I) \neq 0$ and the endomorphism of $\T^N$ induced by $B - I$ is surjective. Choose $v \in \T^N$ with $(B - I) v = -w$. A direct computation shows that conjugating $x \mapsto Bx + w$ by the translation by $v$ gives the automorphism induced by $B$. Hence $g$ is topologically conjugate to that automorphism. An ergodic automorphism of the torus is mixing with respect to Haar measure by \cite{Ha}, and Haar measure has full support, so the automorphism, and therefore $g$, is topologically mixing. The torsion points of $\T^N$ are dense and each subgroup of points of a fixed finite order is finite and invariant under $B$, so every torsion point is periodic for $B$, and density of periodic points follows for $g$ as well. Finally, the topological entropy of the automorphism induced by $B$ equals the sum of $\log|\beta|$ over the eigenvalues $\beta$ of $B$ with $|\beta| > 1$ counted with multiplicity, see \cite{Wb}, and topological entropy is invariant under topological conjugacy.
\end{proof}

We now list the consequences for group actions generated by commuting maps. Two commuting maps generate an action of $\Z^2$, and the statements below apply in particular to that situation.

\begin{proposition}\label{prop:actions}
Let $f \in \mathcal{U}$ and let $\mathcal{G}$ be any set of homeomorphisms of $\T^N$ each of which commutes with $f$. Then every measurable subset of $\T^N$ invariant under $f$ and under every element of $\mathcal{G}$ is either null or conull for Lebesgue measure. If moreover $f$ does not have the accessibility property, then one and the same homeomorphism homotopic to the identity conjugates $f$ to $A$ and every element of $\mathcal{G}$ to an affine transformation in $\Aff_A$.
\end{proposition}

\begin{proof}
Every set invariant under the whole family is in particular invariant under $f$, and $f$ is ergodic with respect to Lebesgue measure by Theorem \ref{thm:au}, which proves the first assertion. The second assertion is contained in Theorem \ref{thm:main}, since the homeomorphism $H$ there depends only on $f$.
\end{proof}

\begin{corollary}\label{cor:classes}
Let $f \in \mathcal{U}$ be without the accessibility property, let $H$ be as in Theorem \ref{thm:main}, and denote by $C(x)$ the accessibility class of $x$ for $f$. Then $H$ maps every accessibility class of $f$ onto the projection to the torus of an affine subspace of $\R^N$ parallel to $E^s \oplus E^u$. Moreover every $g \in Z(f)$ satisfies
\[
g\big( C(x) \big) = C\big( g(x) \big) \qquad \text{for every } x \in \T^N ,
\]
and under $H$ the induced permutation of the family of accessibility classes becomes the permutation of the family of the projected subspaces parallel to $E^s \oplus E^u$ induced by the corresponding affine transformation.
\end{corollary}

\begin{proof}
We work with the lifted objects and the map $H_1$ of Proposition \ref{prop:conjugacy}. Points of a lifted accessibility class have the same value of $h^c$, so $H_1(\widetilde{C}(x)) \subseteq H_1(x) + E^s \oplus E^u$. If two distinct classes were mapped into the same affine subspace parallel to $E^s \oplus E^u$ their points would share the value of $h^c$, contradicting the fact that the fibers of $h^c$ are exactly the classes. Since $H_1$ is a bijection and both the classes and the affine subspaces partition $\R^N$, each subspace is covered by the image of the unique class mapped into it, which gives equality and, after projecting, the first assertion. Let $g \in Z(f)$ and write $H \circ g \circ H^{-1}(x) = Bx + w$ as in Theorem \ref{thm:main}. The matrix $B$ commutes with $A$ and therefore preserves $E^s \oplus E^u$ by Lemma \ref{lem:linearalgebra}, so the affine map $x \mapsto Bx + w$ maps the projection of every affine subspace parallel to $E^s \oplus E^u$ onto another such projection. Applying $H^{-1}$ on both sides gives $g(C(x)) = C(g(x))$ together with the description of the induced permutation.
\end{proof}

The exclusivity of the dichotomy combines with the smooth centralizer classification of Sandfeldt \cite{Sa} to yield a statement in which the two sides of the alternative are separated by centralizers in different categories. For a $C^\infty$ diffeomorphism $f$ we write $Z^\infty(f)$ for the group of $C^\infty$ diffeomorphisms commuting with $f$, and we call $Z^\infty(f)$ virtually trivial when the powers of $f$ generate a subgroup of finite index in it.

\begin{corollary}\label{cor:smooth}
Suppose $N = 4$, or $N = 6$ and $p_A$ is irreducible, and let $\mathcal{U}' \subseteq \mathcal{U}$ be a neighborhood of $A$ small enough in the $C^1$ topology for the results of \cite{Sa} to apply. Let $f \in \mathcal{U}'$ be of class $C^\infty$. Then exactly one of the following holds.
\begin{enumerate}
\item The diffeomorphism $f$ has the accessibility property and $Z^\infty(f)$ is virtually trivial.
\item The diffeomorphism $f$ does not have the accessibility property and $Z(f) \cong \Aff_A$ is virtually free abelian of rank $r_1 + r_2 - 1$.
\end{enumerate}
Moreover $Z^\infty(f)$ fails to be virtually trivial if and only if $f$ is $C^\infty$ conjugate to $A$, and in that case $f$ is not accessible.
\end{corollary}

\begin{proof}
For $N = 4$ the polynomial $p_A$ is irreducible by Lemma \ref{lem:quartic}, and in both cases the center restriction preserves a Euclidean inner product by Lemma \ref{lem:linearalgebra}, so $A$ satisfies the hypotheses of \cite{Sa}. The results of \cite{Sa} then give the following dichotomy for volume preserving $C^\infty$ diffeomorphisms sufficiently $C^1$ close to $A$. Either $Z^\infty(f)$ is virtually trivial, or $f$ is $C^\infty$ conjugate to $A$. Suppose $f$ has the accessibility property. By Proposition \ref{prop:exclusive} the map $f$ is not topologically conjugate to $A$, so in particular it is not $C^\infty$ conjugate to $A$, and the dichotomy forces $Z^\infty(f)$ to be virtually trivial. If instead $f$ does not have the accessibility property, then $Z(f) \cong \Aff_A$ by Theorem \ref{thm:main} and the rank statement is contained in Corollary \ref{cor:structure} and Corollary \ref{cor:centralizerA}. The two cases exclude each other by Theorem \ref{thm:main}.

For the final assertion, if $Z^\infty(f)$ is not virtually trivial then $f$ is $C^\infty$ conjugate to $A$ by the dichotomy of \cite{Sa}, hence not accessible by Proposition \ref{prop:exclusive}. Conversely, if $f$ is $C^\infty$ conjugate to $A$ then $Z^\infty(f)$ is isomorphic to $Z^\infty(A)$, which contains the affine transformations in $\Aff_A$ and is therefore not virtually trivial, since $\Aff_A$ is virtually free abelian of rank $r_1 + r_2 - 1$, which equals two when $N = 4$ by Corollary \ref{cor:structure} and is at least two when $N = 6$, because $r_1 + 2 r_2 = 6$ and $r_2 \geq 1$ give $r_1 + r_2 - 1 = 5 - r_2 \geq 2$.
\end{proof}

Since the linear automorphism $A$ itself belongs to $\mathcal{U}$ and, by Proposition \ref{prop:exclusive} applied with the identity conjugacy, does not have the accessibility property, the second alternative of Theorem \ref{thm:main} is nonvacuous, and applied to $f = A$ with $H$ the identity it recovers the classical description of the centralizer of $A$.

\section{Concluding remarks}

The alternative in Theorem \ref{thm:main} is a genuine dichotomy. The accessibility property is defined through the strong stable and strong unstable foliations and is not in general invariant under topological conjugacy, so the exclusivity of the two cases is not automatic. A perturbative form of the equivalence was noted in \cite{Sa} as a consequence of \cite{RH}. Proposition \ref{prop:exclusive} shows that conjugacy to the linear model constrains the strong foliations of the perturbation through the metric characterization of the stable and unstable sets of $A$, and that this constraint is incompatible with accessibility because $p(E^s \oplus E^u)$ is a proper subgroup of the torus. By Corollary \ref{cor:equivalence} the failure of accessibility within $\mathcal{U}$ is therefore equivalent to topological conjugacy to $A$, and Theorem \ref{thm:main} computes the centralizer on exactly one side of the dichotomy. \\

The conjugacy $H$ is in general only continuous, so the elements $H^{-1} \circ \alpha \circ H$ with $\alpha \in \Aff_A$ are a priori only homeomorphisms, and it is natural to ask which elements of $\Aff_A$ are realized by smooth diffeomorphisms commuting with a given nonlinear $f$ without the accessibility property. For irreducible $p_A$ this question connects the present results with the theorem of Sandfeldt \cite{Sa}, which asserts that a volume preserving diffeomorphism $C^1$ close to $A$ whose smooth centralizer has an image of maximal rank $r_1 + r_2 - 1$ in the automorphism group of homology is $C^\infty$ conjugate to $A$. In the jointly integrable case the topological centralizer always attains the maximal rank by Theorem \ref{thm:main}, so \emph{the gap between the topological and the smooth centralizer of such an $f$ measures precisely the failure of smoothness of $H$}, and one expects the smooth centralizer to be much smaller for many $f$, in analogy with the phenomena established by Damjanovi\'c, Wilkinson and Xu \cite{DWX} and with the program of \cite{DWWX}. On the accessible side the situation is governed by the circle of questions around the problem of Smale \cite{Sm} on the genericity of trivial centralizers, for which the strongest general result available is the theorem of Bonatti, Crovisier and Wilkinson \cite{BCW} in the $C^1$ topology. \\

A refinement of the centralizer question becomes available in the light of Theorem \ref{thm:main}. The classification of \cite{Sa} concerns the smooth centralizer $Z^\infty(f)$ on its own, and in dimensions four and six it shows that $Z^\infty(f)$ is either virtually trivial or of maximal rank, in which case $f$ is smoothly conjugate to $A$, without reference to the group of all commuting homeomorphisms, which was not known at the time. Theorem \ref{thm:main} identifies that larger group in the non accessible case, and every $C^\infty$ diffeomorphism commuting with $f$ is in particular a homeomorphism commuting with $f$, so the smooth centralizer becomes a subgroup $Z^\infty(f) \subseteq Z(f) \cong \Aff_A$. This makes the following question precise. Given a $C^\infty$ diffeomorphism $f$ in the non accessible case, which elements of $\Aff_A$ are realized by $C^\infty$ diffeomorphisms commuting with $f$? Here realized means that the topological conjugacy $H$ of Theorem \ref{thm:main} carries the commuting diffeomorphism to the given affine map, and the point is that $H$ is only continuous, so the smoothness of an element of $Z(f)$ is not automatic even though every affine map is itself smooth. Understanding the structure of the subgroup $Z^\infty(f)$, namely whether it has finite index, whether it is generated by the powers of $f$ and the finite order elements alone, or whether it is cut out by some algebraic condition on the linear parts, is an open and potentially tractable problem that bridges topological and smooth rigidity. \\

The results of Sections 2 and 3 use only ergodicity together with semisimplicity of the center restriction, and Lemma \ref{lem:density} is proved for every ergodic automorphism with no restriction on the dimension of the center space. The restriction to two dimensional center therefore enters only through the alternative of Theorem \ref{thm:au}. In particular the density statement is not an obstacle to ergodic automorphisms with center space of even dimension larger than two, which is precisely the situation excluded by the Salem number condition of \cite{AU} in dimensions six and nine. The remaining obstacles are of a different nature, namely the minimality criterion and the linearization of the holonomy action of $\Z^N$ on a center space of dimension larger than two, where the theorems of Herman and Moser used in \cite{RH} and \cite{AU} are no longer available.
Semisimplicity of the center restriction, which is automatic in dimension two by Lemma \ref{lem:linearalgebra}, also becomes a genuine additional hypothesis for the Lemma \ref{lem:transfer} when the center has dimension larger than two.
It would be interesting to determine, for a nonlinear $f \in \mathcal{U}$ without the accessibility property, the subgroup of $\Aff_A$ consisting of the elements realized by $C^1$ or by $C^r$ diffeomorphisms commuting with $f$. In particular to decide whether this subgroup can be strictly larger than the group generated by the powers of $f$ and the finite order elements. A second interesting concern is the regularity threshold, since the $C^{22}$ hypothesis of \cite{AU} comes from the theorems of Herman and Moser and there is no indication that it is necessary. Therefore any lowering of the regularity in the joint integrability alternative would immediately lower it in Theorem \ref{thm:main} and in Corollary \ref{cor:equivalence}.

\end{document}